\documentclass[11pt,letterpaper]{amsart}

\usepackage[T1]{fontenc}
\usepackage[utf8]{inputenc}
\usepackage{lmodern}
\usepackage{amsmath,amssymb,mathtools}
\usepackage{microtype}
\usepackage{enumitem}
\usepackage[colorlinks=true,linkcolor=blue,citecolor=blue,urlcolor=blue]{hyperref}
\hypersetup{
  pdftitle={Kohayakawa's conjecture and clique coverings of complements of paths, cycles},
  pdfauthor={Bo Ning},
  pdfsubject={Kohayakawa's conjecture and clique coverings of complements of paths, cycles},
  pdfkeywords={clique covering number, forest, bipartite Kneser graph, induced path, Johnson graph, Lovasz local lemma}
}

\allowdisplaybreaks

\newtheorem{theorem}{Theorem}[section]
\newtheorem{proposition}[theorem]{Proposition}
\newtheorem{lemma}[theorem]{Lemma}
\newtheorem{corollary}[theorem]{Corollary}
\newtheorem{conjecture}[theorem]{Conjecture}

\newcommand{\cc}{\operatorname{cc}}

\title[Kohayakawa's conjecture and clique coverings]{Kohayakawa's conjecture and clique coverings of complements of paths and cycles}
\author{Bo Ning}
\address{College of Computer Science, Nankai University, Tianjin 300350, China. Email: bo.ning@nankai.edu.cn}


\begin{document}

\begin{abstract}
For $s\ge1$, let $G_s$ be the bipartite graph between the $s$-subsets and the $(s-1)$-subsets of $[2s]$, where adjacency means disjointness, and let $w(s)$ be the maximum number of $s$-subsets on an induced path in $G_s$. We prove
$w(s)\ge \frac{4^s}{2048s^{5/2}}$
for all $s\geq 6$. This implies $\sup_{s\ge1}w(s)^{1/s}=4$, as conjectured by Kohayakawa (1991). His recursive construction then gives induced paths of order $\Omega(4^r/r^{5/2})$ in the Kneser graph $KG(2r+1,r)$ and yields
\[
 \max\{\cc(\overline{P_n}),\ \cc(\overline{C_n})\}
 \le \log_2 n+\frac52\log_2\log_2 n+O(1).
\]
Together with the known lower bounds, this settles a conjecture of de Caen, Gregory, and Pullman (1985) and gives
\[
 \cc(\overline{P_n})=\log_2 n+\Theta(\log_2\log_2 n),
 \qquad
 \cc(\overline{C_n})=\log_2 n+\Theta(\log_2\log_2 n).
\]
We also give an independent proof of the latter order estimates. It uses a Hamiltonicity result of Kneser graphs and a key lemma proved by the Lov\'{a}sz local lemma. 
\end{abstract}

\maketitle

\section{Introduction}

Throughout this paper, all graphs are finite and simple. For a graph $G$, let $V(G)$ and $E(G)$ denote its vertex and edge sets, respectively. A \emph{clique covering} of $G$ is a family of cliques whose edge sets cover $E(G)$. The minimum cardinality of such a family is the \emph{clique covering number} of $G$, denoted by $\cc(G)$. This parameter is also called the edge clique cover number \cite{Gyarfas1990}. We write $\overline G$ for the complement of $G$. For a nonnegative integer $r$, let $[r]=\{1,\ldots,r\}$, where $[0]=\varnothing$.

Erd\H{o}s, Goodman, and P\'{o}sa~\cite{EGP1966} proved that $\cc(G)$ is the least integer $r$ for which there exists a family $(A_v)_{v\in V(G)}$ of subsets of $[r]$ satisfying
\begin{equation}\label{eq:set-representation-intro}
 uv\in E(G) \quad\Longleftrightarrow\quad A_u\cap A_v\ne\varnothing.
\end{equation}
for all distinct $u,v\in V(G)$. Equivalently, $\cc(\overline G)$ is the least such $r$ for which
\[
 uv\in E(G) \quad\Longleftrightarrow\quad A_u\cap A_v=\varnothing.
\]

Let $P_n$ and $C_n$ denote the path and the cycle on $n$ vertices, respectively. Let $mK_2$ denote a matching of size $m$. In 1985, de Caen, Gregory, and Pullman~\cite{deCaenGregoryPullman1985} determined $\cc(\overline{P_n})$ for $2\le n\le29$ and $\cc(\overline{C_n})$ for $3\le n\le28$ and proved that, for all $n\ge12$,
\begin{equation}\label{eq:dcgp-bounds}
 \log_2n+\frac12\log_2\log_2n
 <\cc(\overline{P_{n+1}})
 \le\cc(\overline{C_n})-2
 \le\cc(\overline{P_{n-1}})
 \le2\log_2n;
\end{equation}
see Theorem~3.2 in~\cite{deCaenGregoryPullman1985}. They conjectured that both covering numbers are asymptotic to $\log_2n$ (see the statement above Section~4 of \cite{deCaenGregoryPullman1985}). In 2005, Cavers~\cite[Conjecture~4.1.5]{Cavers2005} recorded this conjecture as an open problem in his MS.C. Thesis (see Chapter~4).

\begin{conjecture}[de Caen--Gregory--Pullman~\cite{deCaenGregoryPullman1985}]
\label{conj:dcgp}
We have
\[
 \cc(\overline{P_n})\sim\log_2n
 \quad\text{and}\quad
 \cc(\overline{C_n})\sim\log_2n
\]
as $n\to\infty$.
\end{conjecture}

Gy\'{a}rf\'{a}s~\cite[p.~103]{Gyarfas1990} stated that $\cc(\overline{C_n})=\log_2n+o(\log_2n)$ had been proved in~\cite{GregoryPullman1982,deCaenGregoryPullman1985}. However, the latter paper proves only~\eqref{eq:dcgp-bounds} and proposes $\cc(\overline{C_n})\sim\log_2n$ as a conjecture.

After the work of~\cite{deCaenGregoryPullman1985}, several papers made progress toward Conjecture~\ref{conj:dcgp}. Successive improvements in the upper bound for $\cc(\overline{P_n})$ came from induced paths in Kneser graphs. For integers $r\ge2k$, let $KG(r,k)$ be the graph whose vertices are the $k$-subsets of $[r]$, with two vertices adjacent when the corresponding sets are disjoint. Lov\'{a}sz~\cite{Lovasz1978} proved, using algebraic topology, that $\chi(KG(r,k))=r-2k+2$; see also the short proofs of B\'{a}r\'{a}ny~\cite{Barany1978} and Greene~\cite{Greene2002}.  For general background on Kneser graphs, see the monograph of Godsil and Meagher~\cite{GodsilMeagher2016}. Alles and Poljak~\cite{AllesPoljak1989} constructed long induced paths in Kneser graphs. As observed by Kohayakawa~\cite{Kohayakawa1991}, the standard correspondence between induced subgraphs of Kneser graphs and uniform intersection representations then gives
$\cc(\overline{P_n})\le1.695\log_2n$
for all sufficiently large $n$. Kohayakawa~\cite{Kohayakawa1991} reduced the constant to $1.459$ by a recursive construction in bipartite Kneser graphs.

For $s\ge1$, let $G_s$ be the bipartite graph with vertex classes $\binom{[2s]}s$ and $\binom{[2s]}{s-1}$, where two vertices in different classes are adjacent if and only if they are disjoint. Let $w(s)$ denote the maximum number of vertices from $\binom{[2s]}s$ occurring on an induced path in $G_s$. Kohayakawa~\cite{Kohayakawa1991} obtained $w(6)\geq 300$ by a computer search. He pointed out that to prove $f(r)=p(2r+1,r)>(4-o(1))r$ (see the definition in Section 4), one would have to show that $\sup w(s)^{1/s} = 4$. Cavers~\cite[Section~3.3, p.~40]{Cavers2005} attributes to Kohayakawa the conjecture
\begin{equation}\label{eq:kohayakawa-conjecture}
 \sup_{s\ge1}w(s)^{1/s}=4.
\end{equation}
Together with Kohayakawa's recursive construction, this would imply $\cc(\overline{P_n})\le(1+o(1))\log_2n$.

\begin{conjecture}[Kohayakawa \cite{Kohayakawa1991}]\label{conj:kohayakawa}
$\sup_{s\ge1}w(s)^{1/s}=4.$   
\end{conjecture}

We prove~ Conjecture~\ref{conj:kohayakawa} in the following quantitative form: 
\begin{theorem}\label{thm:Kohyakawa-weak}
$w(s)\ge 4^s/(2048s^{5/2})$ for all $s\geq 6$.     
\end{theorem}
Kohayakawa's recursion then gives induced paths of order $\Omega(4^r/r^{5/2})$ in $KG(2r+1,r)$ and the upper bound $\log_2n+\frac52\log_2\log_2n+O(1)$ for both clique covering numbers in Conjecture~\ref{conj:dcgp}. This yields one proof of Theorem~\ref{thm:paths-cycles}.

We settle Conjecture~\ref{conj:dcgp} and determine the order of $\cc(\overline{P_n})-\log_2n$.

\begin{theorem}\label{thm:paths-cycles}
As $n\to\infty$,
\[
 \cc(\overline{P_n})=\log_2n+\Theta(\log_2\log_2n)
 \quad\text{and}\quad
 \cc(\overline{C_n})=\log_2n+\Theta(\log_2\log_2n).
\]
\end{theorem}

We also give a second, independent proof using the Lov\'asz's local lemma, because one key lemma applies more generally to arbitrary set sequences of bounded disjoint-ness degree. The resulting covering lemma may also be useful in related problems.

There is also substantial work on arbitrary forests. For the complement of a perfect matching, Orlin~\cite{Orlin1977} asked for its clique covering number, and Gregory and Pullman~\cite{GregoryPullman1982} determined it exactly. In particular, as $m\to\infty$, they proved that
$\cc(\overline{mK_2})=\log_2m+\frac12\log_2\log_2m+O(1).$
Alon~\cite{Alon1986} proved that if $G$ is an $n$-vertex graph with $\delta(G)\geq 1$ and maximum degree at most $d$, then
$\cc(\overline G)\le2e^2(d+1)^2\ln n.$
Cavers~\cite[Theorem~3.5.2]{Cavers2005} proved
$\cc(\overline F)\le10.3\log_2n$
for every sufficiently large $n$-vertex forest $F$, and conjectured that
\begin{equation}\label{eq:forest-conjecture}
 \cc(\overline F)\le(1+o(1))\log_2n.
\end{equation}
For comparison, the clique partition number of complements of forests was bounded by $O(n\log n)$ in~\cite{CaversVerstraete2008} and subsequently by $O(n\log\log n)$ in~\cite{ConlonFoxSudakov2014}.

Clique coverings have also been studied for graphs with independence
number two~\cite{CharbitEtAl2021} and for uniform
hypergraphs~\cite{RodlSales2024}. The edge clique covering sum and
the local clique covering variant were studied in
\cite{DavoodiJavadiOmoomi2016} and
\cite{JavadiMalekiOmoomi2016}, respectively. Clique coverings of
complete multipartite graphs and their connections with generalized
covering designs were studied in
\cite{BaileyBurgessCaversMeagher2011,DavoodiGerbnerMethukuVizer2020}. The path and cycle parameters also occur as subset indices \cite{McGrew2021}. For a survey on graph covering and partitioning problems, see Schwartz~\cite{Schwartz2022}.

The paper is organized as follows. In Section~\ref{sec:preliminaries}, we give the preliminary results on set representations, Kneser graphs, and
the Lov\'asz local lemma. In Section~\ref{sec:kohayakawa}, we prove Kohayakawa's conjecture. We prove Theorem \ref{thm:kohayakawa}, which is stronger than Theorem \ref{thm:Kohyakawa-weak}. In Section~\ref{sec:consequences}, we derive Theorem~\ref{thm:paths-cycles} from these results and other arguments. In Section~\ref{sec:local-proof}, we give a second proof of the order estimates in Theorem \ref{thm:paths-cycles}  by using the local lemma. We conclude this paper in Section~\ref{sec:remarks} with some remarks.

\section{Preliminaries}\label{sec:preliminaries}

For a graph $G$, let $\Delta(G)$ denote its maximum degree. A vertex set is \emph{stable} if its vertices are pairwise nonadjacent. We use the set-representation theorem in the following form.

\begin{proposition}[Erd\H{o}s--Goodman--P\'{o}sa~\cite{EGP1966}]\label{prop:egp}
Let $G$ be a graph and let $r$ be a nonnegative integer. Then $\cc(G)\le r$ if and only if there are sets $A_v\subseteq[r]$, one for each $v\in V(G)$, such that, for all distinct $u,v\in V(G)$,
\[
 uv\in E(G)\quad\Longleftrightarrow\quad A_u\cap A_v\ne\varnothing.
\]
\end{proposition}

The Kneser graph $KG(2k+1,k)$ is called an \emph{odd graph}. Note that every vertex of this graph has degree $k+1$. We use the following Hamiltonicity theorem.

\begin{theorem}[M\"utze--Nummenpalo--Walczak~\cite{MutzeNummenpaloWalczak2021}]\label{thm:odd-hamiltonian}
For every integer $k\ge3$, the odd graph $KG(2k+1,k)$ has a Hamilton cycle.
\end{theorem}

Our proof of Lemma \ref{lem:repair} uses Lov\'asz's local lemma. Before stating its symmetric form, we need to recall the probability terminology needed below (for more context, see Alon and Spencer~\cite[Chapters~1 and~5]{AlonSpencer2016}). All random variables in Section~\ref{sec:local-proof} are defined on a finite probability space. A Bernoulli random variable with parameter $p$ takes the values $1$ and $0$ with probabilities $p$ and $1-p$, respectively. A graph $\Gamma$ on a finite family $\mathcal B$ of events is a \emph{dependency graph} if each $B\in\mathcal B$ is independent of every intersection of events indexed outside $\{B\}\cup N_\Gamma(B)$.

\begin{lemma}[Erd\H{o}s--Lov\'{a}sz~\cite{ErdosLovasz1975}]\label{lem:lll}
Let $\Gamma$ be a dependency graph for a finite family $\mathcal B$ of events. Suppose that $\Pr(B)\le p$ for every $B\in\mathcal B$ and that $\Delta(\Gamma)\le D$. If
\[
 ep(D+1)\le1,
\]
then
\[
 \Pr\left(\bigcap_{B\in\mathcal B}B^c\right)>0.
\]
\end{lemma}

\section{Kohayakawa's conjecture}\label{sec:kohayakawa}
The main purpose of this section is to prove the following.
\begin{theorem}\label{thm:kohayakawa}
For every integer $k\ge2$,
\[
 w\left(k+\left\lceil\log_2(4k-3)\right\rceil+1\right)
 \ge\binom{2k}{k}.
\]
Moreover, for all $s\geq 6$,
\[
 w(s)\ge\frac{4^s}{2048s^{5/2}}.
\]
Consequently,
\[
 \sup_{s\ge1}w(s)^{1/s}=4.
\]
\end{theorem}

Recall that $G_s$ is the bipartite graph with vertex classes $\binom{[2s]}s$ and $\binom{[2s]}{s-1}$, where two vertices in different classes are adjacent if and only if they are disjoint.  We first prove the following result which is a  criterion for an induced path.

\begin{lemma}\label{lem:middle-layer-criterion}
Let $A_1,\ldots,A_m$ be distinct $s$-subsets of $[2s]$. Suppose that
\[
 |A_i\mathbin\triangle A_{i+1}|=2
 \qquad(1\le i<m)
\]
and
\begin{equation}\label{eq:no-third-set}
 \left\{A_j:A_j\subseteq A_i\cup A_{i+1},\ j\in[m]\right\}
 =\{A_i,A_{i+1}\}
 \qquad(1\le i<m).
\end{equation}
Then $w(s)\ge m$.
\end{lemma}

\begin{proof}
For $1\le i<m$, put
$D_i=[2s]\setminus(A_i\cup A_{i+1}).$
As $|A_i\mathbin\triangle A_{i+1}|=|A_i\setminus A_{i+1}|+|A_{i+1}\setminus A_i|=2$
and $|A_i\setminus A_{i+1}|=|A_{i+1}\setminus A_i|$, we have $|A_i\setminus A_{i+1}|=|A_{i+1}\setminus A_i|=1$.
Thus, $|A_i\cup A_{i+1}|=|A_{i}\setminus A_{i+1}|+|A_{i+1}|=s+1$. It follows that each $D_i$ has cardinality $s-1$. Moreover,
\[
 A_j\cap D_i=\varnothing
 \quad\Longleftrightarrow\quad
 A_j\subseteq A_i\cup A_{i+1}.
\]
Thus,~\eqref{eq:no-third-set} tells us that $D_i$ is adjacent in $G_s$ to exactly $A_i$ and $A_{i+1}$ among the selected $s$-sets $\{A_1,A_2,\ldots,A_m\}.$ 

Moreover, we claim that the sets $D_1,\ldots,D_{m-1}$ are distinct. Indeed, if $D_i=D_j$ ($i\neq j$), then
$A_i\cup A_{i+1}=A_j\cup A_{j+1}.$
Condition~\eqref{eq:no-third-set} gives
\[
 \{A_i,A_{i+1}\}=\{A_j,A_{j+1}\}.
\]
The distinctness of the $A_\ell$ forces $i=j$. Recall $G_s$ is a bipartite graph, and so $D_i,D_j$ are non-adjacent, and $A_i,A_j$ are non-adjacent. Hence
\[
 A_1,D_1,A_2,D_2,\ldots,D_{m-1},A_m
\]
is an induced path in $G_s$.
\end{proof}

For a set $U$ of size $2k$, let $J(U,k)$ be the Johnson graph on $\binom Uk$, where two sets are adjacent when their symmetric difference has size two. Alspach~\cite{Alspach2013} proved that every Johnson graph is
Hamilton-connected. In particular, \(J(U,k)\) has a Hamilton path.
Fix one, say
\begin{equation}\label{eq:payload-path}
 B_1,B_2,\ldots,B_N,
 \qquad N=\binom{2k}{k}
\end{equation}
in $J(U,k)$, and write $e_i=B_iB_{i+1}$ and $C_i=B_i\cup B_{i+1}$ for $1\le i<N$.

Define an auxiliary graph \(F\) as follows. Denote by its vertices set $V(F)=\{e_1,\ldots,e_{N-1}\}$ from above. For each $i$ and each
$B_j\subseteq C_i,$
$B_j\notin\{B_i,B_{i+1}\},$
join $e_i$ to every edge of the path~\eqref{eq:payload-path} incident with $B_j$.

\begin{lemma}\label{lem:conflict-colouring}
The graph $F$ is $(4k-4)$-degenerate. In particular,
\[
 \chi(F)\le4k-3.
\]
\end{lemma}

\begin{proof}
As the set $C_i$ has $(k+1)$ elements, it has $k+1$ subsets of size $k$, of which two are $B_i$ and $B_{i+1}$. Each of the other $k-1$ sets is incident with at most two edges of~\eqref{eq:payload-path}. Consequently, for each $i$, the adjacency rule associated with $C_i$
introduces at most $2(k-1)$ edges of $F$ incident with $e_i$.

Let $X\subseteq V(F)$. By definition, every edge of $F[X]$ is
introduced by the rule associated with at least one of its ends.
Hence
\[
 |E(F[X])|\le 2(k-1)|X|.
\]
Every subgraph of $F$ therefore has average degree at most $4(k-1)$ and contains a vertex of degree at most $4k-4$. This proves the degeneracy assertion. Since every $d$-degenerate graph is $(d+1)$-colorable, we obtain $\chi(F)\le 4k-3$.
\end{proof}

Put $q=4k-3$. By Lemma~\ref{lem:conflict-colouring}, \(F\) admits a proper
coloring
\[
    \gamma:V(F)\to[q].
\]
Next, we shall give a colored binary encoding.
Let
\[
 \ell=\left\lceil\log_2 q\right\rceil,
 \qquad h=\ell+1.
\]
Since \(q\le 2^\ell\), choose an injection
$\alpha:[q]\to\{0,1\}^{\ell},$
and define
\[
 \tau(c)=(\alpha(c),0)\in\{0,1\}^{h}
 \qquad(c\in[q]).
\]
For \(a,b\in[q]\), join \(\tau(a)\) and \(\tau(b)\) by the
following path in the \(h\)-dimensional hypercube \(Q_h\). If $a=b$, take the one-vertex path. If \(a\ne b\), first flip the last coordinate, then flip,
in increasing coordinate order, the coordinates in which
\(\alpha(a)\) and \(\alpha(b)\) differ, and finally flip the
last coordinate again. This gives a simple path whose internal vertices all have the
last coordinate \(1\). In particular, no internal vertex is equal to $\tau(c)$ for any $c\in[q]$.

Let \(W=[h]\times\{0,1\}\) be disjoint from \(U\). For \(z=(z_1,\ldots,z_h)\in\{0,1\}^h\), set
\[
    T(z)=\{(j,z_j):j\in[h]\}.
\]
Thus, \(T(z)\) contains one element from each pair
\(\{(j,0),(j,1)\}\).
If $z$ and $z'$ differ in one coordinate, then $T(z)$ and $T(z')$ are adjacent in $J(W,h)$. Moreover, \(T(z)\cup T(z')\) contains no other set of the
form \(T(y)\), where \(y\in\{0,1\}^h\).

\begin{proof}[\underline{Proof of the first assertion of Theorem~\ref{thm:kohayakawa}}]
Set
\[
 t=k+h=k+\left\lceil\log_2(4k-3)\right\rceil+1.
\]
Since \(|U\cup W|=2k+2h=2t\), we may identify
\(U\cup W\) with \([2t]\). 

We construct a sequence of $t$-subsets of $U\cup W$. 
Start with
\[
 B_1\cup T(\tau(\gamma(e_1))).
\]
For \(1\le i<N\), replace \(B_i\) by \(B_{i+1}\) while
keeping the \(W\)-part fixed; that is, pass from
$B_i\cup T(\tau(\gamma(e_i)))$
to
$B_{i+1}\cup T(\tau(\gamma(e_i))).$
If \(i<N-1\), let \(z_0,z_1,\ldots,z_r\) be the path in
\(Q_h\) defined above, where
\[
    z_0=\tau(\gamma(e_i))
    \quad\text{and}\quad
    z_r=\tau(\gamma(e_{i+1})).
\]
With \(B_{i+1}\) fixed, append
$B_{i+1}\cup T(z_1),\ldots,B_{i+1}\cup T(z_r).$
In this way we obtain a sequence
$A_1,A_2,\ldots,A_M$
of \(t\)-subsets of \(U\cup W\). The sets \(B_1,\ldots,B_N\) are distinct. For each fixed
\(B_i\), the vectors \(z\) used with \(B_i\) are distinct,
since they occur on a simple path in \(Q_h\). Hence all
members of the sequence are distinct.

It remains to verify~\eqref{eq:no-third-set}. First consider the consecutive pair
\[
    B_i\cup T(\tau(\gamma(e_i))),
    \qquad
    B_{i+1}\cup T(\tau(\gamma(e_i))).
\]
The union of these two sets is
$C_i\cup T(\tau(\gamma(e_i))).$
Suppose that a member \(B_j\cup T(z)\) of the sequence is
contained in this union. Since \(U\cap W=\varnothing\),
\[
    B_j\subseteq C_i
    \quad\text{and}\quad
    T(z)\subseteq T(\tau(\gamma(e_i))).
\]
Both sets in the second inclusion have size \(h\), and hence
\[
    T(z)=T(\tau(\gamma(e_i))).
\]
Since the map \(z\mapsto T(z)\) is injective,
$z=\tau(\gamma(e_i)).$
Since no internal vertex of a chosen path in \(Q_h\) belongs
to \(\tau([q])\), there is an edge \(f\) of
\eqref{eq:payload-path} incident with \(B_j\) such that
\[
    z=\tau(\gamma(f)).
\] If \(B_j\notin\{B_i,B_{i+1}\}\), then the definition of \(F\)
implies that \(e_i\) and \(f\) are adjacent in \(F\). Hence
\(\gamma(e_i)\ne\gamma(f)\), since \(\gamma\) is proper.
On the other hand, the injectivity of \(\tau\) gives
\(\gamma(e_i)=\gamma(f)\), a contradiction. Therefore, the only members of the sequence contained in
$C_i\cup T(\tau(\gamma(e_i)))$
are
\[
    B_i\cup T(\tau(\gamma(e_i)))
    \quad\text{and}\quad
    B_{i+1}\cup T(\tau(\gamma(e_i))).
\]

Next consider a consecutive pair
$B_i\cup T(z), B_i\cup T(z'),$
arising from an edge \(zz'\) of one of the chosen paths in
\(Q_h\). If a selected set $B_j\cup T(y)$ is contained in the union of its two ends, then $B_j\subseteq B_i$, and hence $B_j=B_i$. Moreover,
\[
    T(y)\subseteq T(z)\cup T(z').
\]
Since \(z\) and \(z'\) differ in one coordinate, the only
sets of the form \(T(x)\) contained in \(T(z)\cup T(z')\)
are \(T(z)\) and \(T(z')\). Thus, \(B_j\cup T(y)\) is one of the two members of the
consecutive pair.

The sequence $(A_1,\ldots,A_M)$ therefore satisfies the hypotheses
of Lemma~\ref{lem:middle-layer-criterion}. It contains at least one
member with $U$-part $B_i$ for every $i\in[N]$, so $M\ge N$.
Consequently,
\[
w(t)\ge M\ge N=\binom{2k}{k}.
\]
\end{proof}

\begin{lemma}\label{lem:w-monotone}
For every $s\ge1$,
\[
 w(s+1)\ge w(s).
\]
\end{lemma}

\begin{proof}
Let $U$ be the ground set of $G_s$, and let $P$ be an induced path in $G_s$ containing \(w(s)\) vertices from the \(s\)-set class. Choose distinct elements $x,y\notin U$. For every vertex $A$ of $P$, define
\[
\varphi(A)=
\begin{cases}
A\cup\{x\},&\text{if }|A|=s,\\
A\cup\{y\},&\text{if }|A|=s-1.
\end{cases}
\]
Then $\varphi(A)$ is a vertex of $G_{s+1}$. Moreover, if $A$ is an $s$-set and $B$ is an $(s-1)$-set, then
\[
\varphi(A)\cap\varphi(B)
=(A\cup\{x\})\cap(B\cup\{y\})
=A\cap B.
\]
Hence $A$ and $B$ are adjacent in $G_s$ if and only if $\varphi(A)$ and $\varphi(B)$ are adjacent in $G_{s+1}$. Thus $\varphi$ preserves both edges and nonedges among the vertices of $P$, so $\varphi(P)$ is an induced path in $G_{s+1}$ containing
\(w(s)\) vertices from the \((s+1)\)-set class. Therefore $w(s+1)\ge w(s)$.
\end{proof}

\begin{proof}[\underline{Completion of the proof of Theorem~\ref{thm:kohayakawa}}]
For $k\ge2$, define
\[
 t_k=k+\left\lceil\log_2(4k-3)\right\rceil+1.
\]
Since
\[
 4k+1<2(4k-3),
\]
we have
\[
 \left\lceil\log_2(4k+1)\right\rceil
 \le
 \left\lceil\log_2(4k-3)\right\rceil+1.
\]
Consequently,
\[
 1\le t_{k+1}-t_k\le2.
\]

Let $s\ge t_2=6$, and let $k\ge2$ be maximal subject to $t_k\le s$. Then
\[
 t_k\le s<t_{k+1}\le t_k+2.
\]
Since all these quantities are integers, it follows that
\[
 s-t_k\le1.
\]
Furthermore,
\[
\begin{aligned}
 s-k
 &=(s-t_k)+(t_k-k)\le \left\lceil\log_2(4k-3)\right\rceil+2\le \left\lceil\log_2 k\right\rceil+4\\
 &\le \log_2 k+5.
\end{aligned}
\]

By Lemma~\ref{lem:w-monotone} and the first assertion of the theorem,
\[
 w(s)\ge w(t_k)\ge\binom{2k}{k}.
\]
Using
\[
 \binom{2k}{k}\ge\frac{4^k}{2\sqrt{k}}
\]
and
\[
 4^{s-k}\le4^{\log_2 k+5}=1024k^2,
\]
we obtain
\[
\begin{aligned}
 w(s)
 &\ge\frac{4^k}{2\sqrt{k}}=\frac{4^s}{2\sqrt{k}\,4^{s-k}}\ge\frac{4^s}{2048k^{5/2}}\ge\frac{4^s}{2048s^{5/2}},
\end{aligned}
\]
where the last inequality uses $k\le s$.

On the other hand, we have
\[
 w(s)\le \binom{2s}{s}<4^s.
\]
Hence
\[
 \frac{4}{(2048s^{5/2})^{1/s}}
 \le w(s)^{1/s}<4.
\]
Since $(2048s^{5/2})^{1/s}\to1$, it follows that
\[
 \lim_{s\to\infty}w(s)^{1/s}=4.
\]
Finally, $w(s)^{1/s}<4$ for every $s\ge1$, while the sequence converges to $4$. Therefore, we have proved
\[
 \sup_{s\ge1}w(s)^{1/s}=4.
\]
\end{proof}

\section{Consequences for odd graphs and clique coverings}
\label{sec:consequences}

For integers $n\ge2r$, let $p(n,r)$ denote the maximum order of an
induced path in $KG(n,r)$.
Kohayakawa~\cite[Lemma~2]{Kohayakawa1991} established the following
recurrence: for $r\ge2$ and $s\ge1$,
\begin{equation}\label{eq:kohayakawa-recurrence}
 p(2(r+s)+1,r+s)\ge
 \begin{cases}
  w(s)(p(2r+1,r)+1)-1,
      &\text{if $p(2r+1,r)$ is odd},\\
  w(s)p(2r+1,r)-1,
      &\text{if $p(2r+1,r)$ is even}.
 \end{cases}
\end{equation}

We first prove the following corollary.

\begin{corollary}\label{cor:odd-paths}
There is an absolute constant $c>0$ such that
\[
 p(2r+1,r)\ge c\frac{4^r}{r^{5/2}}
\]
for every sufficiently large $r$.
\end{corollary}

\begin{proof}[Proof of Corollary~\ref{cor:odd-paths}]
Applying~\eqref{eq:kohayakawa-recurrence} with the parameters $(r,s)=(2,r-2)$, and using $p(5,2)=5$ from~\cite{Kohayakawa1991}, we obtain
\[
 p(2r+1,r)\ge6w(r-2)-1.
\]
By Theorem~\ref{thm:kohayakawa},
\[
 p(2r+1,r)
 \ge
 \frac{6\cdot4^{r-2}}{2048(r-2)^{5/2}}-1.
\]
For all sufficiently large $r$, the right-hand side is at least
\[
 \frac{3\cdot4^{r-2}}{2048(r-2)^{5/2}}
 \ge
 \frac{3}{16\cdot2048}\frac{4^r}{r^{5/2}}.
\]
The assertion follows with $c=3/(16\cdot2048)$.
\end{proof}

\begin{proof}[First proof of Theorem~\ref{thm:paths-cycles}]
Let $A_1,\ldots,A_m$ be the vertices, in order, of an induced path
in $KG(2r+1,r)$. By the definition of the Kneser graph, two vertices are adjacent
precisely when they are disjoint. Since $A_1,\ldots,A_m$ induce a
path, for distinct $i,j\in[m]$ we have
\[
 A_i\cap A_j=\varnothing
 \quad\Longleftrightarrow\quad
 A_iA_j\in E(KG(2r+1,r))
 \quad\Longleftrightarrow\quad
 |i-j|=1.
\]
Thus, $A_1,\ldots,A_m$ form an intersection representation of
$\overline{P_m}$ on $[2r+1]$. Proposition~\ref{prop:egp} yields
\[
 \cc(\overline{P_m})\le2r+1.
\]

Set $L=\log_2 n$, and fix a constant $C$ such that $c2^C\ge1$. Define
\[
 r=\left\lceil
 \frac12\left(L+\frac52\log_2L+C\right)
 \right\rceil.
\]
For all sufficiently large $n$, we have $r\le L$. By the definition of $r$,
\[
 2r\ge L+\frac52\log_2L+C,
\]
and hence
\[
 4^r=2^{2r}
 \ge2^{L+\frac52\log_2L+C}
 =2^C nL^{5/2}.
\]
Corollary~\ref{cor:odd-paths} therefore gives
\[
 p(2r+1,r)
 \ge c\frac{4^r}{r^{5/2}}
 \ge c2^C n
 \ge n.
\]
Hence $KG(2r+1,r)$ contains an induced path of order $n$, and therefore
\[
\begin{aligned}
 \cc(\overline{P_n})
 &\le2r+1\le
 \log_2n+\frac52\log_2\log_2n+C+3\\
 &=\log_2n+\frac52\log_2\log_2n+O(1).
\end{aligned}
\]

By de Caen, Gregory, and
Pullman~\cite[Theorem~3.2]{deCaenGregoryPullman1985}, we have
\begin{equation}\label{eq:path-to-cycle}
 \cc(\overline{C_n})
 \le \cc(\overline{P_{n-1}})+2.
\end{equation}
Applying the preceding bound with $n-1$ in place of $n$, we obtain
\[
 \cc(\overline{C_n})
 \le
 \log_2n+\frac52\log_2\log_2n+O(1).
\]
Together with the lower bounds in~\eqref{eq:dcgp-bounds}, these
estimates prove the theorem.
\end{proof}
\section{A direct local-lemma proof}\label{sec:local-proof}

\subsection{A stable-set covering lemma}

The following lemma covers the edges of a bounded-degree graph by
$O(\log d)$ subsets that are stable in a prescribed path. Here the symbol ``$\ln$"
denotes the natural logarithm.

\begin{lemma}\label{lem:repair}
Let $P=x_1x_2\cdots x_n$ be a path, and let $H$ be a graph on $V(P)$
such that
$E(H)\cap E(P)=\varnothing$
and $\Delta(H)\le d,$
where $d\ge1$. Then there are stable sets $I_1,\ldots,I_s$ in $P$
such that every edge of $H$ has both ends in some $I_j$, and
$s\le\left\lceil64\ln(10ed)\right\rceil.$
\end{lemma}

\begin{proof}
We adapt Alon's two stage random-clique
construction~\cite[Lemma~3.2]{Alon1986} to the complement of a path.
In place of Alon's first-moment argument, we use the Lov\'asz local
lemma to obtain a bound depending on $\Delta(H)$ rather than on $n$.

Set
$s=\left\lceil64\ln(10ed)\right\rceil.$
For each $i\in[n]$ and $j\in[s]$, choose $X_{i,j}\in\{0,1\}$
independently and uniformly at random. Set
\[
 X_{0,j}=X_{n+1,j}=0
\]
and define
\[
 I_j=\{x_i:X_{i,j}=1,\ X_{i-1,j}=X_{i+1,j}=0\}.
\]
No two consecutive vertices can belong to $I_j$, so $I_j$ is
stable in $P$.

Fix $x_ax_b\in E(H)$ with $a<b$. Since
$E(H)\cap E(P)=\varnothing$, we have $b-a\ge2$. For each $j\in[s]$,
the event $\{x_a,x_b\}\subseteq I_j$ requires
$X_{a,j}=X_{b,j}=1$ and
\[
 X_{a-1,j}=X_{a+1,j}=X_{b-1,j}=X_{b+1,j}=0,
\]
after repeated variables and boundary variables are omitted. These
requirements are consistent. If $b=a+2$, then the only repeated
variable is $X_{a+1,j}=X_{b-1,j}$, which is required to be $0$ in
both places. Hence
\begin{equation}\label{eq:one-repair-probability}
 \Pr(\{x_a,x_b\}\subseteq I_j)\ge2^{-6}.
\end{equation}

Let $\mathcal B_{ab}$ be the event that no $I_j$ contains both
$x_a$ and $x_b$. The events
\[
 \{x_a,x_b\}\subseteq I_j,
 \qquad j\in[s],
\]
are independent, since they depend on disjoint families of random
variables. Thus
\[
 \Pr(\mathcal B_{ab})
 \le(1-2^{-6})^s
 \le e^{-s/64}
 \le\frac1{10ed}.
\]

The event $\mathcal B_{ab}$ is determined by the variables $X_{i,j}$
with $j\in[s]$ and
\[
 i\in\{a-1,a,a+1,b-1,b,b+1\}\cap[n].
\]
Define a graph $\Gamma$ on the bad events by joining
$\mathcal B_{ab}$ and $\mathcal B_{cd}$ if an end of $x_cx_d$ is at
distance at most two in $P$ from an end of $x_ax_b$. For each bad
event $\mathcal B_{ab}$, the variables determining
$\mathcal B_{ab}$ are disjoint from those determining its
nonneighbors in $\Gamma$. Since the underlying variables are
mutually independent, $\mathcal B_{ab}$ is independent of every
event determined by its nonneighbors. Thus $\Gamma$ is a dependency
graph.

Let $S$ be the union of the closed radius-two neighborhoods of
$x_a$ and $x_b$ in $P$. Then $|S|\le10$. Every neighbor of
$\mathcal B_{ab}$ in $\Gamma$ corresponds to an edge of
$H$ other than $x_ax_b$ having an end in $S$. Hence
\[
\begin{aligned}
 d_\Gamma(\mathcal B_{ab})
 &\le
 \bigl|\{e\in E(H):e\text{ has an end in }S\}\bigr|-1\le\sum_{v\in S}d_H(v)-1\le10d-1.
\end{aligned}
\]
Therefore, $\Delta(\Gamma)\le10d-1.$
The symmetric Lov\'asz local lemma applies with
$p=\frac1{10ed}$
and $D=10d-1,$
since
\[
 ep(D+1)\le
 e\cdot\frac1{10ed}\cdot10d=1.
\]
Thus, with positive probability, none of the events
$\mathcal B_{ab}$ occurs. The resulting stable sets
$I_1,\ldots,I_s$ satisfy the assertion.
\end{proof}

The lemma has the following set-theoretic consequence.

\begin{corollary}\label{cor:repair}
Let $U$ be a finite set, let $d$ be a positive integer, and let
$B_1,\ldots,B_n$ be nonempty subsets of $U$. Suppose that
\begin{enumerate}[label=\textup{(\roman*)}]
 \item $B_i\cap B_{i+1}=\varnothing$ for $1\le i<n$;
 \item
 $\left|\{j\in[n]\setminus\{i\}:B_i\cap B_j=\varnothing\}\right|
  \le d$
for every $i\in[n].$
\end{enumerate}
Then there is a set $U'\supseteq U$ with
$|U'|\le |U|+\left\lceil64\ln(10ed)\right\rceil$
and subsets $A_1,\ldots,A_n$ of $U'$ such that, for all distinct
$i,j\in[n]$,
\[
 A_i\cap A_j=\varnothing
 \quad\Longleftrightarrow\quad
 |i-j|=1.
\]
\end{corollary}

\begin{proof}
Let $P=x_1\cdots x_n$, and define a graph $H$ on $V(P)$ by
\[
 x_ix_j\in E(H)
 \quad\Longleftrightarrow\quad
 |i-j|\ge2\ \text{and}\ B_i\cap B_j=\varnothing.
\]
Then $E(H)\cap E(P)=\varnothing$, and condition~\textup{(ii)} gives
$\Delta(H)\le d$. Let $I_1,\ldots,I_s$ be the stable sets given by
Lemma~\ref{lem:repair}. Choose distinct elements
$y_1,\ldots,y_s\notin U$, set
$U'=U\cup\{y_1,\ldots,y_s\},$
and define
\[
 A_i=B_i\cup\{y_j:x_i\in I_j\}
 \qquad(i\in[n]).
\]

If $|i-j|=1$, then $B_i\cap B_j=\varnothing$ by
condition~\textup{(i)}. Since every $I_\ell$ is stable in $P$, no
$I_\ell$ contains both $x_i$ and $x_j$. Hence
$A_i\cap A_j=\varnothing$.

Now suppose that $|i-j|\ge2$. If $B_i\cap B_j\ne\varnothing$, then
$A_i\cap A_j\ne\varnothing$. Otherwise $x_ix_j\in E(H)$, so
Lemma~\ref{lem:repair} gives an $\ell\in[s]$ such that
$x_i,x_j\in I_\ell$. Consequently,
$y_\ell\in A_i\cap A_j.$
This proves the corollary.
\end{proof}

\subsection{The main estimate}

We apply Corollary~\ref{cor:repair} to an initial segment of a
Hamilton cycle in an odd graph.

\begin{proposition}\label{prop:path-upper}
Let $n$ be a positive integer, and let $k\ge3$ satisfy
\[
 n\le\binom{2k+1}{k}.
\]
Then
\[
 \cc(\overline{P_n})
 \le
 2k+1+\left\lceil64\ln\bigl(10e(k+1)\bigr)\right\rceil.
\]
\end{proposition}

\begin{proof}
By Theorem~\ref{thm:odd-hamiltonian}, $KG(2k+1,k)$ has a Hamilton
cycle. Let $B_1,\ldots,B_n$ be $n$ consecutive vertices of this
cycle. Since adjacency in a Kneser graph is defined by disjointness,
\[
 B_i\cap B_{i+1}=\varnothing
 \qquad(1\le i<n).
\]
Moreover, every $k$-subset of $[2k+1]$ is disjoint from exactly
\[
 \binom{k+1}{k}=k+1
\]
other $k$-subsets. Thus, each $B_i$ is disjoint from at most $k+1$
of the sets $B_1,\ldots,B_n$.

Apply Corollary~\ref{cor:repair} with $U=[2k+1]$ and $d=k+1$. It
gives a set $U'\supseteq[2k+1]$ with
\[
 |U'|
 \le
 2k+1+\left\lceil64\ln\bigl(10e(k+1)\bigr)\right\rceil
\]
and subsets $A_1,\ldots,A_n$ of $U'$ such that, for distinct
$i,j\in[n]$,
\[
 A_i\cap A_j=\varnothing
 \quad\Longleftrightarrow\quad
 |i-j|=1.
\]
Thus, the intersection graph of $A_1,\ldots,A_n$ is
$\overline{P_n}$. Proposition~\ref{prop:egp} completes the proof.
\end{proof}

\begin{proof}[Second proof of the order estimates in
Theorem~\ref{thm:paths-cycles}]
Let $k\ge3$ be the least integer such that
\begin{equation}\label{eq:k-choice}
 n\le\binom{2k+1}{k}.
\end{equation}
For all sufficiently large $n$, we have $k\ge4$, and hence the
minimality of $k$ gives
\[
 n>\binom{2k-1}{k-1}.
\]
Since
\[
 2^{2k-1}
 =\sum_{i=0}^{2k-1}\binom{2k-1}{i}
 \le2k\binom{2k-1}{k-1},
\]
we have
\[
 n>\frac{2^{2k-1}}{2k}.
\]
Taking logarithms gives
\begin{equation}\label{eq:k-upper}
 2k+1<\log_2n+\log_2(2k)+2.
\end{equation}

Since $2k\le2^k$ for $k\ge3$,
\[
 n>\frac{2^{2k-1}}{2k}\ge2^{k-1}.
\]
On the other hand,
\[
 n\le\binom{2k+1}{k}\le2^{2k+1}.
\]
Consequently,
\[
 k=\Theta(\log_2n).
\]
In particular,
\[
 \left\lceil64\ln\bigl(10e(k+1)\bigr)\right\rceil
 =O(\log_2\log_2n).
\]
Proposition~\ref{prop:path-upper} and~\eqref{eq:k-upper} now yield
\begin{equation}\label{eq:path-upper-asymptotic}
 \cc(\overline{P_n})
 \le\log_2n+O(\log_2\log_2n).
\end{equation}

By~\eqref{eq:path-to-cycle},
\begin{equation}\label{eq:cycle-upper-asymptotic}
 \cc(\overline{C_n})
 \le\cc(\overline{P_{n-1}})+2
 \le\log_2n+O(\log_2\log_2n).
\end{equation}

The leftmost inequalities in~\eqref{eq:dcgp-bounds} give
\[
 \cc(\overline{C_n})-\log_2n
 =\Omega(\log_2\log_2n).
\]
Applying the leftmost inequality in~\eqref{eq:dcgp-bounds} with
$n-1$ in place of $n$, we also obtain
\[
 \cc(\overline{P_n})-\log_2n
 =\Omega(\log_2\log_2n).
\]
Together with~\eqref{eq:path-upper-asymptotic}
and~\eqref{eq:cycle-upper-asymptotic}, these inequalities prove the
theorem.
\end{proof}

\section{Concluding remarks}\label{sec:remarks}
In this paper, we show that for every integer $k\ge2$,
$$w\left(k+\left\lceil\log_2(4k-3)\right\rceil+1\right)
 \ge\binom{2k}{k}.$$ We do not know whether there exists an absolute constant \(C\geq 0\) such
that
\[
 w(k+C)\ge\binom{2k}{k}
\]
for every sufficiently large \(k\).

Theorem~\ref{thm:paths-cycles} leaves a gap between the coefficients $1/2$ and $5/2$ in the second-order term. In particular, for each of the following two sequences, it remains open
whether it converges as $n\to\infty$ and, if it does, what its limit is:
\[
 \frac{\cc(\overline{P_n})-\log_2 n}{\log_2\log_2 n}
 \quad\text{and}\quad
 \frac{\cc(\overline{C_n})-\log_2 n}{\log_2\log_2 n}.
\]
\section*{Declaration on the Use of AI}
During the preparation of this work, the author used AI systems to assist in generating candidate
proof strategies, particularly in identifying the function used in the proof of Theorem \ref{thm:kohayakawa}. All AI-generated
suggestions were verified and refined by the author, who takes full responsibility for the correctness and
originality of the paper.

\begingroup
\small
\linespread{1.0}\selectfont

\endgroup

\end{document}